\documentclass{amsart}

\usepackage{graphicx}
\usepackage{amsmath}
\usepackage{amsfonts}
\usepackage{amssymb}
\usepackage{bbm}
\usepackage{epic}
\usepackage{amscd}
\usepackage{amsthm}

\usepackage[margin=0.75in]{geometry}

\usepackage[all]{xy}

\usepackage{multirow}
\usepackage{tikz-cd}
\usepackage{url}

\newtheorem{theorem}{Theorem}[section]

\theoremstyle{definition}

\DeclareMathOperator{\diag}{diag}

\newcommand{\CP}{\mathbbm{CP}}

\newcommand{\Complex}{\mathbbm{C}}

\newcommand{\calF}{{\mathcal F}}

\newcommand{\calE}{{\mathcal E}}

\newcommand{\calO}{{\mathcal O}}

\DeclareMathOperator{\Aut}{Aut}

\DeclareMathOperator{\Jac}{Jac}

\newcommand{\pitilde}{{\tilde{\pi}}}

\begin{document}


\title{The moduli space of stable rank 2 parabolic bundles over an
  elliptic curve with 3 marked points}

\author{David Boozer} 

\begin{abstract}
We explicitly describe the moduli space $M^s(X,3)$ of stable
rank 2 parabolic bundles over an elliptic curve $X$ with trivial
determinant bundle and 3 marked points.
Specifically, we exhibit $M^s(X,3)$ as a blow-up of an embedded
elliptic curve in $(\CP^1)^3$.
The moduli space $M^s(X,3)$ can also be interpreted as the $SU(2)$
character variety of the 3-punctured torus.
Our description of $M^s(X,3)$ reproduces the known Poincar\'{e}
polynomial for this space.
\end{abstract}

\date{\today}

\maketitle

\section{Introduction}

Given a curve $C$, one can define a moduli space $M^s(C,n)$ of stable
rank 2 parabolic bundles over $C$ with trivial determinant bundle and
$n$ marked points.
The space $M^s(C,n)$ has the structure of a smooth complex
manifold of dimension $3(g-1) + n$, where $g$ is the genus of the curve $C$.
In general, the space $M^s(C,n)$ depends on a positive real parameter
$\mu$ known as the \emph{weight}.
For $\mu$ sufficiently small ($\mu < 1/n$ will suffice), the space
$M^s(C,n)$ is independent of $\mu$, but as $\mu$ increases it may
cross critical values at which $M^s(C,n)$ undergoes certain birational
transformations \cite{Boden,Thaddeus}.
The moduli space $M^s(C,n)$ can also be interpreted as an
$SU(2)$-character variety, which is defined as the space of conjugacy
classes of $SU(2)$-representations of the fundamental group of
$C$ with $n$ punctures, where loops around the punctures are required
to correspond to $SU(2)$-matrices conjugate to
$\diag(e^{2\pi i \mu}, e^{-2\pi i\mu})$.
Moduli spaces of parabolic bundles on curves are natural objects of
study in algebraic geometry, and also play an important role in
low-dimensional topology.
In particular, these spaces have a canonical symplectic structure and
can be used to define Floer homology theories of links
\cite{Boozer-2,Hedden-1,Hedden-2,Horton}.

Explicit descriptions of $M^s(C,n)$ are known for small values of $n$
and $C$ a rational or elliptic curve.
For rational curves, it is well-known that for small weight we have
\begin{align*}
  M^s(\CP^1,0) &= M^s(\CP^1,1) = M^s(\CP^1,2) = \varnothing, \\
  M^s(\CP^1,3) &= \{pt\}, \\
  M^s(\CP^1,4) &= \CP^1 - \{\textup{3 points}\}, \\
  M^s(\CP^1,5) &= \CP^2 \# 4 \overline{\CP}^2.
\end{align*}
The structure of $M^s(\CP^1,6)$ for $\mu = 1/4$, corresponding to the
traceless character variety, was recently described by Kirk
\cite{Kirk}.
For an elliptic curve $X$, it is straightforward to show that for
small weight we have
\begin{align*}
  M^s(X,0) &= \varnothing, &
  M^s(X,1) &= \CP^1.
\end{align*}
It was recently shown by Vargas \cite{Vargas} that $M^s(X,2)$ is the
complement of an embedded elliptic curve in $(\CP^1)^2$.

Our goal in this paper is to explicitly describe the structure of
$M^s(X,3)$.
We prove the following result:

\begin{theorem}
\label{theorem:intro-main}
For small weight, the moduli space $M^s(X,3)$ for an elliptic curve
$X$ is a blow-up of an embedded elliptic curve in $(\CP^1)^3$.
\end{theorem}

To prove Theorem \ref{theorem:intro-main}, we make use of an explicit
description of Hecke modifications of rank 2 holomorphic vector
bundles on elliptic curves that is derived in \cite{Boozer}.
Roughly speaking, a Hecke modification is a way of locally
modifying a vector bundle near a point to obtain a new vector bundle.
The moduli space $M^s(X,3)$ plays an important role in a conjectural
Floer homology theory for links in lens spaces discussed in
\cite{Boozer}, and Theorem \ref{theorem:intro-main} was motivated by
this application.

The paper is organized as follows.
In Sections
\ref{sec:parabolic} and \ref{sec:vector},
we review the background material we will need on parabolic bundles
and vector bundles on elliptic curves.
In Section
\ref{sec:description},
we use Hecke-modification methods to explicitly
describe $M^s(X,3)$.
In Section
\ref{sec:relationship},
we relate $M^s(X,3)$ to $M^{ss}(X,2)$, the moduli space
of $S$-equivalence classes of rank 2 semistable parabolic bundles with
trivial determinant bundle and 2 marked points.
In Section
\ref{sec:poincare},
we use our description of $M^s(X,3)$ to reproduce the known
Poincar\'{e} polynomial for this space.

\section{Parabolic bundles}
\label{sec:parabolic}

The concept of a parabolic bundle was introduced in
\cite{Mehta-Seshadri}.
We will not need this concept in its full generality; rather, we will
consider only parabolic bundles of a certain restricted form, which is
discussed at greater length in \cite[Appendix B]{Boozer}.
For our purposes here, a rank $2$ parabolic bundle over a curve $C$
consists of a rank 2 holomorphic vector bundle $E$ over $C$ with
trivial determinant bundle, distinct marked points
$p_1, \cdots, p_n \in C$, a
line $\ell_{p_i} \in \mathbbm{P}(E_{p_i})$ in the fiber $E_{p_i}$ over
each marked point $p_i$, and a positive real parameter $\mu$ known as
the \emph{weight}.
For simplicity, we will suppress the curve $C$ and weight $\mu$ in the
notation and denote a parabolic bundle as
$(E,\ell_{p_1},\cdots,\ell_{p_n})$.

In order to describe the stability properties of parabolic bundles, it
is helpful to introduce some additional terminology.
Recall that the degree of the proper subbundles of a vector
bundle $E$ on a curve $C$ is bounded above.
Given a rank 2 holomorphic vector bundle $E$, we say that a line
$\ell_p \in \mathbbm{P}(E_p)$ is \emph{bad} if there is a line
subbundle $L$ of $E$ of maximal degree such that $\ell_p = L_p$, and
\emph{good} otherwise.
We say that lines
$\ell_{p_1} \in \mathbbm{P}(E_{p_1}), \cdots,
\ell_{p_n} \in \mathbbm{P}(E_{p_n})$ are
\emph{bad in the same direction} if there is a
line subbundle $L$ of $E$ of maximal degree such that
$\ell_{p_i} = L_{p_i}$ for $i=1, \cdots, n$.

Consider a parabolic bundle
$\calE = (E,\ell_{p_1},\cdots,\ell_{p_n})$.
Let $m$ denote the maximum number of lines of $\calE$ that are bad in
the same direction.
For sufficiently small weight ($\mu < 1/n$ will suffice), we
can characterize the stability and semistability of $\calE$ as
follows.
If $E$ is unstable, then $\calE$ is unstable.
If $E$ is semistable, then $\calE$ is stable if $m < n/2$, semistable
if $m \leq n/2$, and unstable if $m > n/2$.
Note that if $n$ is odd then stability and semistability are
equivalent.

We define a moduli space $M^s(C,n)$ of isomorphism classes of stable
parabolic bundles with $n$ marked points.
As with vector bundles, one can define a notion of $S$-equivalent
semistable parabolic bundles, and we define a moduli space
$M^{ss}(C,n)$ of $S$-equivalence classes of semistable parabolic
bundles.
For odd $n$ we have that $M^s(C,n) = M^{ss}(C,n)$.
For even $n$ we have that $M^s(C,n)$ is an open subset of
$M^{ss}(C,n)$.

\section{Vector bundles on elliptic curves}
\label{sec:vector}

Vector bundles on elliptic curves were classified by Atiyah
\cite{Atiyah}, and are well understood.
Here we briefly summarize the results regarding vector bundles on
elliptic curves that we will need; these results are either well-known
(see for example \cite{Atiyah,Iena,Teixidor,Tu}) or derived in
\cite[Section 5]{Boozer}.

\subsection{Line bundles}

Isomorphism classes of degree 0 line bundles on an elliptic curve $X$
are parameterized by the Jacobian $\Jac(X)$.
Given a basepoint $e \in X$, we define the \emph{Abel-Jacobi}
isomorphism $X \rightarrow \Jac(X)$, $p \mapsto [\calO(p-e)]$.
Given points $p, e \in X$, we define a \emph{translation map}
$\tau_{p-e}:\Jac(X) \rightarrow \Jac(X)$,
$[L] \mapsto [L \otimes \calO(p-e)]$.
We say that a line bundle $L$ is \emph{2-torsion} if $L^2 = \calO$.
There are four 2-torsion line bundles, which we
denote $L_i$ for $i=1,2,3,4$.

\subsection{Semistable rank 2 vector bundles}
\label{sec:vector-bundles}

For our purposes here, we need only consider semistable rank 2 vector
bundles on an elliptic curve $X$ with trivial determinant bundle.
There are three classes of such bundles.

First, we have vector bundles of the form $E = L \oplus L^{-1}$, where
$L$ is a degree 0 line bundle such that $L^2 \neq \calO$.
There are two bad lines $L_p, (L^{-1})_p \in \mathbbm{P}(E_p)$ in the
fiber $E_p$ over a point $p \in X$, and all other lines in
$\mathbbm{P}(E_p)$ are good.
The automorphism group $\Aut(E)$ of $E$ consists of $GL(2,\Complex)$
matrices of the form
\begin{align*}
  \left(\begin{array}{cc}
    A & 0 \\
    0 & D
  \end{array}\right).
\end{align*}
Each bad line $L_p, (L^{-1})_p \in \mathbbm{P}(E_p)$ is fixed by the
automorphisms of $E$,
and there is a unique (up to rescaling by a constant) automorphism
carrying any good line $\ell_p \in \mathbbm{P}(E_p)$ to any other good
line $\ell_p' \in \mathbbm{P}(E_p)$.

Second, we have four vector bundles of the form
$E = L_i \oplus L_i$, where $L_i$ is a 2-torsion line bundle.
All lines $\ell_p \in \mathbbm{P}(E_p)$ in the fiber $E_p$ over a
point $p \in X$ are bad.
The automorphism group $\Aut(E)$ of $E$ is $GL(2,\Complex)$, and
there is a unique (up to rescaling by a constant) automorphism
carrying any triple of lines
$(\ell_{p_1}, \ell_{p_2}, \ell_{p_3}) \in
\mathbbm{P}(E_{p_1}) \times \mathbbm{P}(E_{p_2}) \times
\mathbbm{P}(E_{p_3})$
such that no two lines are bad in the same direction to any other
triple of lines
$(\ell_{p_1}', \ell_{p_2}', \ell_{p_3}') \in
\mathbbm{P}(E_{p_1}) \times \mathbbm{P}(E_{p_2}) \times
\mathbbm{P}(E_{p_3})$
such that no two lines are bad in the same direction.

Third, we have four vector bundles of the form
$E = F_2 \otimes L_i$,
where $L_i$ is a 2-torsion line bundle and $F_2$ is the unique
non-split extension of $\calO$ by $\calO$:
\begin{eqnarray*}
\begin{tikzcd}
  0 \arrow{r} &
  \calO \arrow{r} &
  F_2 \arrow{r} &
  \calO \arrow{r} &
  0.
\end{tikzcd}
\end{eqnarray*}
There is a unique bad $(L_i)_p \in \mathbbm{P}(E_p)$ in the fiber
$E_p$ over a point $p \in X$, and all other lines in
$\mathbbm{P}(E_p)$ are good.
The automorphism group $\Aut(E)$ of $E$ consists of $GL(2,\Complex)$
matrices of the form
\begin{align*}
  \left(\begin{array}{cc}
    A & B \\
    0 & A
  \end{array}\right).
\end{align*}
The bad line $(L_i)_p \in \mathbbm{P}(E_p)$ is fixed by the
automorphisms of $E$,
and there is a unique (up to rescaling by a constant) automorphism
carrying any good line $\ell_p \in \mathbbm{P}(E_p)$ to any other good
line $\ell_p' \in \mathbbm{P}(E_p)$.

We define $M^{ss}(X)$ to be the moduli space of $S$-equivalence
classes of semistable rank 2 vector bundles on $X$ with trivial
determinant bundle.
In \cite{Tu} it is shown that $M^{ss}(X)$ is isomorphic to $\CP^1$, as
can be understood as follows.
The above classification shows that we can parameterize semistable
rank 2 vector bundles with trivial determinant bundle as
$L \oplus L^{-1}$ for $[L] \in \Jac(X)$, together with the four
bundles $F_2 \otimes L_i$.
The bundles $L \oplus L^{-1}$ and $L^{-1} \oplus L$ are isomorphic,
hence $S$-equivalent, and are thus identified in $M^{ss}(X)$.
One can show that the bundles $F_2 \otimes L_i$ and $L_i \oplus L_i$
are $S$-equivalent, and are thus identified in $M^{ss}(X)$.
It follows that $M^s(X)$ is the quotient of $\Jac(X)$
by the involution $[L] \mapsto [L^{-1}]$, which yields a space known
as the \emph{pillowcase} that is isomorphic to $\CP^1$.
We define a map $p:\Jac(X) \rightarrow M^{ss}(X)$,
$[L] \mapsto [L \oplus L^{-1}]$, which is a branched double-cover with
four branch points $p([L_i]) \in M^{ss}(X)$ corresponding to four
ramification points $[L_i] \in \Jac(X)$ that are fixed by the
involution $[L] \mapsto [L^{-1}]$.

\subsection{Hecke modifications of rank 2 vector bundles}

Given a rank 2 vector bundle $E$ over a curve $C$, distinct points
$p_1, \cdots, p_n \in C$, and lines
$\ell_{p_i} \in \mathbbm{P}(E_{p_i})$ for each point $p_i$, one can
perform a \emph{Hecke modification} of $E$ at each point $p_i$ using
data provided by the line $\ell_{p_i}$ so as to obtain a new vector
bundle that we will denote $H(E,\ell_{p_1},\cdots,\ell_{p_n})$.
One way to describe $H(E,\ell_{p_1},\cdots,\ell_{p_n})$ is as follows.
Let $\calE$ denote the sheaf of sections of $E$, and define a subsheaf
$\calF$ of $\calE$ whose set of sections over an open subset $U$ of
$X$ is given by
\begin{align*}
  \calF(U) = \{s \in \calE(U) \mid
  \textup{$p_i \in U \implies s(p_i) \in \ell_{p_i}$ for
    $i=1,\cdots,n$}\}.
\end{align*}
We define $H(E,\ell_{p_1},\cdots,\ell_{p_n})$ to be the vector bundle
whose sheaf of sections is $\calF$.

Hecke modifications of rank 2 vector bundles on elliptic curves are
described explicitly in \cite{Boozer}.
For our purposes here, all the results we will need can be
described in terms of a few properties of a certain map $h_e$, which
we define as follows.
Given a semistable rank 2 vector bundle $E$ with trivial determinant
bundle over an elliptic curve $X$, distinct points $p, q, e \in X$
such that $p + q = 2e$, and lines
$\ell_p \in \mathbbm{P}(E_p)$ and
$\ell_q \in \mathbbm{P}(E_q)$ that are
not bad in the same direction, we define
$h_e(E,\ell_p, \ell_q) \in M^{ss}(X)$ as
\begin{align*}
  h_e(E,\ell_p,\ell_q) = [H(E,\ell_p,\ell_q) \otimes \calO(e)].
\end{align*}
One can show that if the lines $\ell_p$ and $\ell_q$ are not bad in
the same direction, then $H(E,\ell_p,\ell_q) \otimes \calO(e)$ is in
fact a semistable vector bundle with trivial determinant bundle and
thus represents a point in $M^{ss}(X)$.
It is clear that the order of the lines doesn't matter in the
definition of $H(E,\ell_p,\ell_q)$, so
\begin{align*}
  h_e(E,\ell_p,\ell_q) = h_e(E,\ell_q,\ell_p).
\end{align*}
If $(E,\ell_p,\ell_q)$ and $(E',\ell_p',\ell_q')$ are isomorphic
parabolic bundles, one can show that
\begin{align*}
  h_e(E,\ell_p,\ell_q) =
  h_e(E',\ell_p',\ell_q').
\end{align*}
In particular, for $\phi \in \Aut(E)$ we  have
\begin{align*}
  h_e(E,\ell_p,\ell_q) =
  h_e(E,\phi(\ell_p),\phi(\ell_q)).
\end{align*}
If the line $\ell_p$ is good and $E \neq L_i \oplus L_i$, then
we have the following result:

\begin{theorem}[{\cite[Lemma 5.26]{Boozer}}]
\label{theorem:he-iso}
If $E = L \oplus L^{-1}$ for $L^2 \neq \calO$ or
$E = F_2 \otimes L_i$, and
$\ell_p \in \mathbbm{P}(E_p)$ is a good line, then the map
$\mathbbm{P}(E_q) \rightarrow M^{ss}(X)$,
$\ell_q \mapsto h_e(E,\ell_p,\ell_q)$ is an isomorphism.
\end{theorem}

If the line $\ell_p$ is bad, then $h_e(E,\ell_p,\ell_q)$ is
uniquely determined by $E$ and the points $p$ and $e$:

\begin{theorem}[{\cite[Lemma 5.27]{Boozer}}]
\label{theorem:he-eval}
If $\ell_p \in \mathbbm{P}(E_p)$ is a bad line and
$\ell_q \in \mathbbm{P}(E_q)$ is not bad in the same direction as
$\ell_p$, then $h_e(E,\ell_p,\ell_q)$ is given by
\begin{align*}
  &h_e(L \oplus L^{-1},L_p,\ell_q) =
  (p \circ \tau_{p-e})([L]), &
  &h_e(L \oplus L^{-1},(L^{-1})_p,\ell_q) =
  (p \circ \tau_{e-p})([L]), \\
  &h_e(F_2 \otimes L_i,(L_i)_p,\ell_q) =
  (p \circ \tau_{p-e})([L_i]), &
  &h_e(L_i \oplus L_i,\ell_p,\ell_q) =
  (p \circ \tau_{p-e})([L_i]).
\end{align*}
\end{theorem}

Note that
$(p \circ \tau_{p-e})([L_i]) = (p \circ \tau_{e-p})([L_i])$.
Note also that in Theorem \ref{theorem:he-eval} the line
$\ell_q \in \mathbbm{P}(E_q)$ is allowed to be bad, just not bad in
the same direction as $\ell_p \in \mathbbm{P}(E_p)$.
For example, we have
\begin{align*}
  h_e(L \oplus L^{-1}, L_p, L^{-1}_q) =
  (p \circ \tau_{p-e})([L]) =
  (p \circ \tau_{e-q})([L]).
\end{align*}

\section{Description of $M^s(X,3)$}
\label{sec:description}

We consider here the moduli space $M^s(X,3)$ of stable rank 2
parabolic bundles over an elliptic curve $X$ with trivial determinant
bundle and 3 marked points $p_1, p_2, p_3 \in X$.
If $[E,\ell_{p_1},\ell_{p_2},\ell_{p_3}] \in M^s(X,3)$,
then $E$ is semistable and no two of the lines
$\ell_{p_1},\ell_{p_2},\ell_{p_3}$ are bad in the same direction.
It follows that $E$ has one of the three forms described in
Section \ref{sec:vector-bundles};
that is, $E = L \oplus L^{-1}$ for $L^2 \neq \calO$,
$E = L_i \oplus L_i$, or $E = F_2 \otimes L_i$.
Choose points $e_1, e_2, e_3 \in X$ such that
\begin{align*}
  p_1 + p_2 &= 2e_3, &
  p_3 + p_1 &= 2e_2, &
  p_2 + p_3 &= 2e_1.
\end{align*}
Define a map
$\pi = (\pi_1,\pi_2,\pi_3):M^s(X,3) \rightarrow (M^{ss}(X))^3$ by
\begin{align*}
  \pi([E,\ell_{p_1},\ell_{p_2},\ell_{p_3}]) =
  ([E],\,h_{e_2}(E,\ell_{p_1},\ell_{p_3}),\,h_{e_1}(E,\ell_{p_2},\ell_{p_3})).
\end{align*}
Note that since $E$ is semistable and no two of the lines $\ell_{p_1}$,
$\ell_{p_2}$, $\ell_{p_3}$ are bad in the same direction, we can in
fact define $h_{e_2}(E,\ell_{p_1},\ell_{p_3})$ and
$h_{e_1}(E,\ell_{p_2},\ell_{p_3})$.
It is useful to decompose $M^s(X,3)$ as
\begin{align*}
  M^s(X,3) =
  \{\textup{$\ell_{p_3}$ good}\} \cup
  \{\textup{$\ell_{p_3}$ bad}\},
\end{align*}
where the open submanifold
$\{\textup{$\ell_{p_3}$ good}\}$ and the closed submanifold
$\{\textup{$\ell_{p_3}$ bad}\}$ consist of points
$[E,\ell_{p_1},\ell_{p_2},\ell_{p_3}] \in M^s(X,3)$ for which
$\ell_{p_3}$ is a good and bad line, respectively.
The open submanifold $\{\textup{$\ell_{p_3}$ good}\}$ is described by
the following result:

\begin{theorem}
\label{theorem:subset-good}
The restriction of $\pi:M^s(X,3) \rightarrow (M^{ss}(X))^3$ to
$\{\textup{$\ell_{p_3}$ good}\} \rightarrow
\pi(\{\textup{$\ell_{p_3}$ good}\})$ is an
isomorphism.
\end{theorem}

\begin{proof}
If
$[E,\ell_{p_1},\ell_{p_2},\ell_{p_3}] \in
\{\textup{$\ell_{p_3}$ good}\}$,
then $E$ must have good lines, hence
$E = L \oplus L^{-1}$ for $L^2 \neq \calO$ or $E = F_2 \otimes L_i$.
For each point $[E] \in M^{ss}(X)$, choose a representative $E$ of
$[E]$ and a good line $\ell_{p_3}' \in \mathbbm{P}(E_{p_3})$.
We can define a map
$\pi_1^{-1}([E]) \cap \{\textup{$\ell_{p_3}$ good}\} \rightarrow
\mathbbm{P}(E_{p_1}) \times \mathbbm{P}(E_{p_2})$,
\begin{align*}
  [E,\ell_{p_1},\ell_{p_2},\ell_{p_3}] \mapsto
  (\phi(\ell_{p_1}), \phi(\ell_{p_2})),
\end{align*}
where $\phi$ is the unique (up to rescaling by a constant)
automorphism of $E$ such that $\phi(\ell_{p_3}) = \ell_{p_3}'$.
This map is an isomorphism onto its image, hence by Theorem
\ref{theorem:he-iso} the map
$(\pi_2,\pi_3):
\pi_1^{-1}([E]) \cap \{\textup{$\ell_{p_3}$ good}\} \rightarrow
(M^{ss}(X))^2$ is an
isomorphism onto its image.
\end{proof}

Next we consider the closed submanifold
$\{\textup{$\ell_{p_3}$ bad}\}$.
Elements of $\{\textup{$\ell_{p_3}$ bad}\}$ have one of three forms:
\begin{align*}
  [L \oplus L^{-1}, \ell_{p_1}, \ell_{p_2}, L_{p_3}], &&
  [F_2 \otimes L_i, \ell_{p_1}, \ell_{p_2}, (L_i)_{p_3}], &&
  [L_i \oplus L_i, \ell_{p_1}, \ell_{p_2}, \ell_{p_3}],
\end{align*}
where $L^2 \neq \calO$ and $L_i$ is a 2-torsion line bundle.
Note that elements of the form
$[L \oplus L^{-1}, \ell_{p_1}, \ell_{p_2}, (L^{-1})_{p_3}]$ can
be converted into the first of the three listed forms by applying the
isomorphism $\phi:L \oplus L^{-1} \rightarrow L^{-1} \oplus L$:
\begin{align*}
  [L \oplus L^{-1}, \ell_{p_1}, \ell_{p_2}, (L^{-1})_{p_3}] =
  [L^{-1} \oplus L, \phi(\ell_{p_1}), \phi(\ell_{p_2}),
    (L^{-1})_{p_3}] =
  [M \oplus M^{-1}, \phi(\ell_{p_1}), \phi(\ell_{p_2}), M_{p_3}],
\end{align*}
where we have defined $M = L^{-1}$ and used the fact that
$\phi((L^{-1})_{p_3}) = (L^{-1})_{p_3}$.

Recall that we defined a map
$\pi_1:M^s(X,3) \rightarrow M^{ss}(X)$,
$[E, \ell_{p_1}, \ell_{p_2}, \ell_{p_3}] \mapsto [E]$.
We can lift
$\pi_1:\{\textup{$\ell_{p_3}$ bad}\} \rightarrow M^{ss}(X)$ to the
branched double-cover $p:\Jac(X) \rightarrow M^{ss}(X)$ by using the
bad line $\ell_{p_3}$ to distinguish between distinct vector bundles
$L \oplus L^{-1}$ and $L^{-1} \oplus L$ that are identified in
$M^{ss}(X)$:
\begin{eqnarray*}
\begin{tikzcd}
  {} & 
  \Jac(X) \arrow{d}{p} \\
  \{\textup{$\ell_{p_3}$ bad}\} \arrow{ru}{\pitilde_1} \arrow{r}{\pi_1} &
  M^{ss}(X),
\end{tikzcd}
\end{eqnarray*}
where
$\pitilde_1:\{\textup{$\ell_{p_3}$ bad}\} \rightarrow \Jac(X)$ is
defined such that
\begin{align*}
  \pitilde_1([L \oplus L^{-1}, \ell_{p_1}, \ell_{p_2}, L_{p_3}])
  &= [L], &
  \pitilde_1([F_2 \otimes L_i, \ell_{p_1}, \ell_{p_2}, (L_i)_{p_3}])
  &=
  \pitilde_1([L_i \oplus L_i, \ell_{p_1}, \ell_{p_2}, \ell_{p_3}])
  = [L_i].
\end{align*}
Define a map $f:\Jac(X) \rightarrow (M^{ss}(X))^3$,
$f = (p,\,p \circ \tau_{p_3 - e_2},\,p \circ \tau_{p_3 - e_1})$.

\begin{theorem}
We have a commutative diagram
\begin{eqnarray*}
\begin{tikzcd}
  \{\textup{$\ell_{p_3}$ bad}\} \arrow{dr}{\pi}
  \arrow{d}[swap]{\pitilde_1} &
  {} \\
  \Jac(X) \arrow{r}{f} &
   (M^{ss}(X))^3.
\end{tikzcd}
\end{eqnarray*}
\end{theorem}

\begin{proof}
The fiber of $\pitilde_1$ over a point
$[L] \in \Jac(X)$ such that $L^2 \neq \calO$ is
\begin{align*}
  \pitilde_1^{-1}([L]) = 
  \{[L \oplus L^{-1}, \ell_{p_1}, \ell_{p_2}, L_{p_3}] \}.
\end{align*}
From Theorem \ref{theorem:he-eval} it follows that
$\pi(\pitilde_1^{-1}([L])) = f([L])$.
The fiber of $\pitilde_1$ over the point $[L_i] \in \Jac(X)$ is
\begin{align*}
  \pitilde_1^{-1}([L_i]) = 
  \{[F_2 \otimes L_i, \ell_{p_1}, \ell_{p_2}, (L_i)_{p_3}] \} \cup
  \{[L_i \oplus L_i, \ell_{p_1}, \ell_{p_2}, \ell_{p_3}] \}.
\end{align*}
From Theorem \ref{theorem:he-eval} it follows that
\begin{align*}
  \pi([F_2 \otimes L_i, \ell_{p_1}, \ell_{p_2}, (L_i)_{p_3}]) =
  \pi([L_i \oplus L_i, \ell_{p_1}, \ell_{p_2}, \ell_{p_3}]) =
  f([L_i]).
\end{align*}
Thus $\pi(\pitilde_1^{-1}([L_i])) = f([L_i])$.
\end{proof}

\begin{theorem}
The map $f:\Jac(X) \rightarrow (M^{ss}(X))^3$ is injective.
\end{theorem}

\begin{proof}
Take $[L], [M] \in \Jac(X)$ such that $f([L]) = f([M])$.
Projecting onto the first factor of $(M^{ss}(X))^3$, we find that
either $[M] = [L]$ or $[M] = [L^{-1}]$.
If $[M] = [L^{-1}]$, then projecting onto the second factor of
$(M^{ss}(X))^3$ gives either
$[L \otimes \calO(p_3-e_2)] = [L^{-1} \otimes \calO(p_3-e_2)]$ or
$[L \otimes \calO(p_3-e_2)] = [L \otimes \calO(e_2-p_3)]$.
In the first case $[L] = [L^{-1}]$.
The second case cannot actually occur, since otherwise
$2p_3 = 2e_2 = p_3 + p_1$ and thus $p_1 = p_3$, contradiction.
\end{proof}

Given a point
$[E,\ell_{p_1},\ell_{p_2},\ell_{p_3}] \in
\{\textup{$\ell_{p_3}$ bad}\}$,
we have that $E$ is semistable and $\ell_{p_1}$ and $\ell_{p_2}$
cannot be bad in the same direction, hence we can define a map
$h:\{\textup{$\ell_{p_3}$ bad}\} \rightarrow M^{ss}(X)$,
\begin{align*}
  h([E, \ell_{p_1}, \ell_{p_2}, \ell_{p_3}]) =
  h_{e_3}(E, \ell_{p_1}, \ell_{p_2}).
\end{align*}

\begin{theorem}
\label{theorem:subset-bad}
The map
$(\pitilde_1,h):\{\textup{$\ell_{p_3}$ bad}\} \rightarrow
\Jac(X) \times M^{ss}(X)$ is an isomorphism.
\end{theorem}

\begin{proof}
The fiber of $\pitilde_1$ over a point $[L] \in \Jac(X)$ such that
$L^2 \neq \calO$ is
\begin{align*}
  \pitilde_1^{-1}([L]) = 
  \{[L \oplus L^{-1}, \ell_{p_1}, \ell_{p_2}, L_{p_3}] \}.
\end{align*}
We will argue that $\pitilde_1^{-1}([L])$ is isomorphic to $\CP^1$.
Choose a local trivialization of $E := L \oplus L^{-1}$ over an open
set containing $p_1$ and $p_2$ so as to obtain identifications
$\psi_i:\mathbbm{P}(E_{p_i}) \rightarrow \CP^1$ for $i=1,2$.
We can choose the local trivialization such that
$\psi_i(L_{p_i}) = \infty$ and $\psi_i((L^{-1})_{p_i}) = 0$.
Define $z_i = \psi_i(\ell_{p_i})$ and note that
$(z_1, z_2) \in \Complex^2 - \{(0,0)\}$.
An automorphism of $E$ induces the transformation
$(z_1, z_2) \mapsto a(z_1,z_2)$ for $a \in \Complex^\times$, hence
we have an isomorphism
\begin{align*}
  \pitilde_1^{-1}([L]) =
  \{[L \oplus L^{-1}, \ell_{p_1}, \ell_{p_2}, L_{p_3}] \} \rightarrow
  \CP^1, &&
     [L \oplus L^{-1}, \ell_{p_1}, \ell_{p_2}, L_{p_3}] \mapsto
     [z_1 : z_2].
\end{align*}
A canonical version of this statement is that the restriction of
$h:\{\textup{$\ell_{p_3}$ bad}\} \rightarrow M^{ss}(X)$
to $\pitilde_1^{-1}([L])$ gives an isomorphism
$\pitilde_1^{-1}([L]) \rightarrow M^{ss}(X)$.
In particular, from Theorems \ref{theorem:he-iso} and
\ref{theorem:he-eval} it follows that
\begin{align*}
  &h(\{[L \oplus L^{-1}, \ell_{p_1}, \ell_{p_2}, L_{p_3}] \mid
  \ell_{p_1} \neq (L^{-1})_{p_1},\,
  \ell_{p_2} \neq (L^{-1})_{p_2}\}) =
  M^{ss}(X) -
  \{
  (p \circ \tau_{e_3 - p_1})([L]),\,
  (p \circ \tau_{e_3 - p_2})([L])\}, \\
  &h(\{[L \oplus L^{-1}, \ell_{p_1}, \ell_{p_2}, L_{p_3}] \mid
  \ell_{p_1} = (L^{-1})_{p_1},\,
  \ell_{p_2} \neq (L^{-1})_{p_2}\}) =
  \{
  (p \circ \tau_{e_3 - p_1})([L])\}, \\
  &h(\{[L \oplus L^{-1}, \ell_{p_1}, \ell_{p_2}, L_{p_3}] \mid
  \ell_{p_1} \neq (L^{-1})_{p_1},\,
  \ell_{p_2} = (L^{-1})_{p_2}\}) =
  \{
  (p \circ \tau_{e_3 - p_2})([L])\}.
\end{align*}

The fiber of $\pitilde_1$ over the point $[L_i] \in \Jac(X)$ is
\begin{align*}
  \pitilde_1^{-1}([L_i]) = 
  \{[F_2 \otimes L_i, \ell_{p_1}, \ell_{p_2}, (L_i)_{p_3}] \} \cup
  \{[L_i \oplus L_i, \ell_{p_1}, \ell_{p_2}, \ell_{p_3}] \}.
\end{align*}
Note that
$\{[L_i \oplus L_i, \ell_{p_1}, \ell_{p_2}, \ell_{p_3}] \}$
consists of a single point, since there is a unique (up to rescaling
by a constant) automorphism of
$L_i \oplus L_i$ that induces an isomorphism of any pair of stable
parabolic bundles
$(L_i \oplus L_i, \ell_{p_1}, \ell_{p_2}, \ell_{p_3})$ and
$(L_i \oplus L_i, \ell_{p_1}', \ell_{p_2}', \ell_{p_3}')$.
We will argue that
$\{[F_2 \otimes L_i, \ell_{p_1}, \ell_{p_2}, (L_i)_{p_3}] \}$ is
isomorphic to $\Complex$.
Choose a local trivialization of $E := F_2 \otimes L_i$ over an open
set containing $p_1$ and $p_2$ so as to obtain identifications
$\psi_i:\mathbbm{P}(E_{p_i}) \rightarrow \CP^1$ for $i=1,2$.
We can choose the local trivialization such that
$\psi_i((L_i)_{p_i}) = \infty$.
Define $z_i = \psi_i(\ell_{p_i})$ and note that
$(z_1, z_2) \in \Complex^2$.
An automorphism of $E$ induces the transformation
$(z_1, z_2) \mapsto (z_1 + b, z_2 + b)$ for $b \in \Complex$, hence
we have an isomorphism
\begin{align*}
  \{[F_2 \otimes L_i, \ell_{p_1}, \ell_{p_2}, L_{p_3}] \} \rightarrow
  \Complex, &&
   [F_2 \otimes L_i, \ell_{p_1}, \ell_{p_2}, L_{p_3}] \mapsto
     z_2 - z_1.
\end{align*}
A canonical version of these results is that the restriction of
$h:\{\textup{$\ell_{p_3}$ bad}\} \rightarrow M^{ss}(X)$
to $\pitilde_1^{-1}([L_i])$ gives an isomorphism
$\pitilde_1^{-1}([L_i]) \rightarrow M^{ss}(X)$.
In particular, from Theorems \ref{theorem:he-iso} and
\ref{theorem:he-eval} it follows that
\begin{align*}
  &h(\{[F_2 \otimes L_i, \ell_{p_1}, \ell_{p_2}, L_{p_3}]\}) =
  M^{ss}(X) -
  \{(p \circ \tau_{p_1 - e_3})([L_i]) \}, \\
  &h(\{[L_i \oplus L_i, \ell_{p_1}, \ell_{p_2}, \ell_{p_3}] \}) =
  \{ (p \circ \tau_{p_1 - e_3})([L_i]) \}.
\end{align*}
Note that
$(p \circ \tau_{p_1 - e_3})([L_i]) =
(p \circ \tau_{e_3 - p_1})([L_i]) =
(p \circ \tau_{p_2 - e_3})([L_i])$.
\end{proof}

Theorems \ref{theorem:subset-good}--\ref{theorem:subset-bad} prove
Theorem \ref{theorem:intro-main} from the Introduction.

\section{Relationship between $M^s(X,3)$ and $M^{ss}(X,2)$}
\label{sec:relationship}

In \cite{Vargas} it is shown that $M^{ss}(X,2)$ is isomorphic to
$(\CP^1)^2$.
From our perspective, we can describe this result by defining a map
$M^{ss}(X,2) \rightarrow (M(X)^{ss})^2$,
\begin{align*}
  [E,\ell_{p_1}, \ell_{p_2}] \mapsto
  ([E],\,h_{e_3}(E,\ell_{p_1},\ell_{p_2})).
\end{align*}
One can show that this map is an isomorphism.

We can relate the closed subset $\{\textup{$\ell_{p_3}$ bad}\}$ of
$M^s(X,3)$ to the moduli space $M^{ss}(X,2)$ as follows.
Define a map $\{\textup{$\ell_{p_3}$ bad}\} \rightarrow M^{ss}(X,2)$,
\begin{align*}
  [E, \ell_{p_1}, \ell_{p_2}, \ell_{p_3}] \mapsto
  [E, \ell_{p_1}, \ell_{p_2}].
\end{align*}
We have a commutative diagram
\begin{eqnarray*}
\begin{tikzcd}
  \{\textup{$\ell_{p_3}$ bad}\} \arrow{r}
  \arrow{d}[swap]{\pitilde_1} &
  M^{ss}(X,2) \arrow{d} \\
  \Jac(X) \arrow{r}{p} &
  M^{ss}(X),
\end{tikzcd}
\end{eqnarray*}
where we have defined a map
$M^{ss}(X,2) \rightarrow M^{ss}(X)$,
$[E,\ell_{p_1},\ell_{p_2}] \mapsto [E]$.

\section{Poincar\'{e} polynomial of $M^s(X,3)$}
\label{sec:poincare}

The Poincar\'{e} polynomial of $M^s(C,n)$ is given in
\cite[Theorem 3.8]{Street} for the case $\mu=1/4$, corresponding to
the traceless character variety, and $n$ odd:
\begin{align}
  \label{eqn:pt-c-n}
  P_t(M^s(C,n)) =
  \frac{(1 + t^2)^n (1 + t^3)^{2g} -
    2^{n - 1} t^{2g + n - 1} (1 + t)^{2g} (1 + t^2)}
  {(1 - t^2)(1 - t^4)},
\end{align}
where $g$ is the genus of $C$.
In fact, the results of \cite{Street} are stated for parabolic bundles
with fixed determinant bundle of odd degree, but since
$\mu=1/4$, corresponding to a traceless character variety, the results
also hold for the moduli space $M^s(C,n)$ for
which the determinant bundle of the parabolic bundles is trivial.
For an elliptic curve $X$ with $3$ marked points, equation
(\ref{eqn:pt-c-n}) gives
\begin{align}
  \label{eqn:pt-x-3}
  P_t(M^s(X,3)) &= 1 + 4t^2 + 2t^3 + 4t^4 + t^6.
\end{align}

We can reproduce equation (\ref{eqn:pt-x-3}) using our explicit
description of $M^s(X,3)$.
Since $\pi:M^s(X,3) \rightarrow (M^{ss}(X))^3$ restricts to an
isomorphism
$M^s(X,3) - \{\textup{$\ell_{p_3}$ bad}\} \rightarrow
(M^{ss}(X))^3 - \pi(\{\textup{$\ell_{p_3}$ bad}\})$,
we obtain the following equation for the Poincar\'{e} polynomials for
cohomology with compact supports:
\begin{align}
  \label{eqn:pt-1}
  P_t(M^s(X,3) - \{\textup{$\ell_{p_3}$ bad}\}) =
  P_t((M^{ss}(X))^3 - \pi(\{\textup{$\ell_{p_3}$ bad}\})).
\end{align}
From the long exact sequence for cohomology with compact supports, we
have
\begin{align}
  \label{eqn:pt-2}
  P_t(M^s(X,3) - \{\textup{$\ell_{p_3}$ bad}\}) &=
  P_t(M^s(X,3)) - P_t(\{\textup{$\ell_{p_3}$ bad}\}), \\
  \label{eqn:pt-3}
  P_t((M^{ss}(X))^3 - \pi(\{\textup{$\ell_{p_3}$ bad}\})) &=
  P_t((M^{ss}(X))^3) - P_t(\pi(\{\textup{$\ell_{p_3}$ bad}\})).
\end{align}
We have that $M^{ss}(X)$ is isomorphic to $\CP^1$,
$\pi(\{\textup{$\ell_{p_3}$ bad}\})$ isomorphic to $\Jac(X)$, and
$\{\textup{$\ell_{p_3}$ bad}\}$ is isomorphic to
$\Jac(X) \times M^{ss}(X)$, so
\begin{align}
  \label{eqn:pt-4}
  P_t(M^{ss}(X)^3) &= (1 + t^2)^3, &
  P_t(\pi(\{\textup{$\ell_{p_3}$ bad}\})) &= 1 + 2t + t^2, &
  P_t(\{\textup{$\ell_{p_3}$ bad}\}) &= (1 + 2t + t^2)(1 + t^2).
\end{align}
Combining equations (\ref{eqn:pt-1})--(\ref{eqn:pt-4}), we reproduce
equation (\ref{eqn:pt-x-3}) for $P_t(M^s(X,3))$.

\bibliographystyle{abbrv}
\bibliography{ms-x-3}

\end{document}